\documentclass{amsproc}

\usepackage{amssymb}

\usepackage{amsfonts}

\newtheorem{theorem}{Theorem}[section]
\newtheorem{lemma}[theorem]{Lemma}

\theoremstyle{definition}
\newtheorem{definition}[theorem]{Definition}

\theoremstyle{remark}
\newtheorem{remark}[theorem]{Remark}

\newtheorem{proposition}[theorem]{Proposition}
\newtheorem{corollary}[theorem]{Corollary}

\numberwithin{equation}{section}

\begin{document}

\title[Pseudofinite groups]{Pseudofinite groups with NIP theory and definability in finite simple groups}


\author{Dugald Macpherson}
\address{School of Mathematics\\ University of Leeds\\Leeds LS2
9JT\\UK}
\email{h.d.macpherson@leeds.ac.uk}
\thanks{Research partially supported by EPSRC grant EP/H00677X/1}

\author{Katrin Tent}
\address{Mathematisches Institut\\ 
Universit\"at M\"unster\\ 
Einsteinstrasse 62\\
48149 M\"unster\\ 
Germany}
\email{tent@math.uni-muenster.de}
\thanks{Research partially supported by SFB 878}
\dedicatory{For R\"udiger G\"obel, in celebration of his seventieth birthday.}

\subjclass[2000]{Primary 03C60; Secondary 20D06}

\date{July 29 2011, and, in revised form, February 15 2012}
\keywords{pseudofinite group, NIP theory, word map}

\begin{abstract}
We show that any pseudofinite group with NIP  theory and with a finite upper bound on the length of chains of centralisers is soluble-by-finite. In particular, any NIP rosy pseudofinite group is soluble-by-finite. This generalises, and shortens the proof of, an earlier result for stable pseudofinite groups. An example is given of an NIP pseudofinite group which is not soluble-by-finite. However, if $\mathcal{C}$ is a class of finite groups such that all infinite ultraproducts of members of $\mathcal{C}$ have NIP theory, then there is a bound on the index of the soluble radical of any member of $\mathcal{C}$. We also survey some ways in which model theory gives information on families of finite simple groups, particularly concerning products of images of word maps.

\end{abstract}

\maketitle

  \parskip 1mm

\newcommand{\Ind}{
 \setbox0=\hbox{$x$}\kern\wd0\hbox to 0pt{\hss$
 \mid$\hss}\lower.9\ht0\hbox to 0pt{\hss$\smile$\hss}\kern\wd0
}
\newcommand{\indep}[3]{
 #1\mathop{\mathpalette\Ind{}}_{#2}#3
}
\newcommand{\Indep}{\indep{}{}{}}

\newcommand{\Notind}{
 \setbox0=\hbox{$x$}\kern\wd0\hbox to 0pt{\mathchardef
 \nn=12854\hss$\nn$\kern1.4\wd0\hss}\hbox to
0pt{\hss$\mid$\hss}\lower.9\ht0
 \hbox to 0pt{\hss$\smile$\hss}\kern\wd0
}
\newcommand{\depen}[3]{
 #1\mathop{\mathpalette\Notind{}}_{#2}#3
}

\newcommand\eq{{\rm eq}}
\newcommand\acl{{\rm acl}}
\newcommand\SL{{\rm SL}}
\newcommand\Ff{{\mathbb F}}
\newcommand\Rr{{\mathbb R}}
\newcommand\tp{{\rm tp}}
\newcommand\Th{{\rm Th}}
\newcommand\Soc{{\rm Soc}}
\newcommand\Alt{{\rm Alt}}
\newcommand\Aut{{\rm Aut}}
\newcommand\SU{{\rm SU}}

\newcommand\C{{\mathcal C}}
\newcommand\K{{\mathcal K}}

\def\Ind#1#2{#1\setbox0=\hbox{$#1x$}\kern\wd0\hbox to 0pt{\hss$#1\mid$\hss}
\lower.9\ht0\hbox to 0pt{\hss$#1\smile$\hss}\kern\wd0}
\def\dnf{\mathop{\mathpalette\Ind{}}}
\def\Notind#1#2{#1\setbox0=\hbox{$#1x$}\kern\wd0\hbox to 0pt{\mathchardef
\nn=12854\hss$#1\nn$\kern1.4\wd0\hss}\hbox to
0pt{\hss$#1\mid$\hss}\lower.9\ht0 \hbox to
0pt{\hss$#1\smile$\hss}\kern\wd0}
\def\df{\mathop{\mathpalette\Notind{}}}

\section{Introduction}

We consider in this paper groups $G$ which are {\em pseudofinite}, that is, infinite groups which satisfy every first order sentence (in the language $L_g$ of groups) which holds in all finite groups. Equivalently, $G$ is elementarily equivalent to an infinite ultraproduct of finite groups. Or equivalently again, $G$ is an infinite group with the {\em finite model property}\/: every sentence in the theory of the group has a finite model.
We consider the structure of $G$, under the assumption that the first order theory $\Th(G)$ of $G$ satisfies various generalisations of model theoretic stability.

It was shown in \cite{mactent} that any stable pseudofinite group $G$ has a definable soluble normal subgroup of finite index. This is not surprising; for by a classification due to Wilson \cite{wilson1}
(with a slight strengthening due to Ryten -- see \cite[Proposition 2.14]{ejmr}) -- any infinite pseudofinite {\em simple} group
is a group of Lie type over a pseudofinite field, and in particular interprets a pseudofinite field \cite[5.2.4, 5.3.3, 5.4.3]{ryten}, and so has unstable theory by Duret \cite{duret}. However, an intricate argument with centralisers was needed in \cite{mactent} to bound the derived length of soluble normal subgroups.

One generalisation of stability is the notion of {\em simple} theory. Pseudofinite fields (and certain difference fields, that is, fields equipped with a specified automorphism) are simple, in fact supersimple of finite rank, and it follows from Wilson's classification that every simple pseudofinite group is interpretable in such a structure. Hence, every simple pseudofinite group has supersimple finite rank theory; this follows from the results of Hrushovski \cite{hrushovski} and is made explicit in \cite{ejmr} (note that measurable structures are supersimple of finite rank -- see e.g. \cite[Corollary 3.7]{em}). A satisfactory structure theory for pseudofinite groups with supersimple finite rank theory -- under an additional and probably unnecessary assumption that $\exists^{\infty}$ is definable in $T^{\eq}$ -- was initiated in \cite{ejmr}. The class of supersimple finite rank structures is sufficiently rich to include a lot of pseudofinite group theory, as indicated by, for example, \cite[4.11, 4.12]{lmt}. Possible applications of the model theory of supersimple theories to finite simple groups are discussed in the final section of the present paper.

Another generalisation of stability of considerable current interest is that of {\em NIP}, or {\em dependent} theory. A formula $\phi(\bar{x},\bar{y})$ has the {\em independence property} with respect to  $T$
if there is $M\models T$ and a set $\{\bar{a}_i:i\in \omega\}\subset M^{l(\bar{x})}$ such that for all 
$S\subseteq \omega$ there is $\bar{b}_S\in M^{l(\bar{y})}$ such that for all $i\in \omega$,
$M\models \phi(\bar{a}_i,\bar{b}_S)$ if and only if $i\in S$. A theory $T$ is {\em NIP} if no formula has the independence property with respect to $T$. Any stable theory is simple and NIP, and any theory which is {\em both} simple and NIP is stable. For groups, by the Baldwin-Saxl Theorem (see \cite{saxl}, or \cite[Fact 0.17]{ekp}) the NIP condition implies a useful chain condition: if $G$ is an NIP group, then for every formula $\phi(x,\bar{y})$ there is a natural number $n_\phi$ such that every {\em finite} intersection of $\phi$-definable groups is an intersection of $n_\phi$ $\phi$-definable groups. By Wilson's theorem, there is no  simple  pseudofinite group with NIP theory, and we expected this, together with the above chain condition, to yield virtual solubility for pseudofinite groups with NIP theory. However, this is false, and in Section 3 below we give a construction of a pseudofinite group $G$ with NIP theory which is not soluble-by-finite.

Our main theorem is the following. We say that a group $G$ has the {\em centraliser chain condition} if there is a natural number $n=n(G)$ such that there do not exist subsets $F_1,\ldots,F_{n+1}\subset G$ with
$$C_G(F_1)<\ldots<C_G(F_{n+1}).$$

\begin{theorem} \label{main}
Let $G$ be a pseudofinite group with NIP theory, and suppose that $G$ satisfies the centraliser chain condition. Then $G$ has  a soluble definable normal subgroup of finite index.
\end{theorem}

We obtain some information about finite groups just under an NIP assumption. Let us say that the class $\mathcal{C}$ of finite structures is an {\em NIP class} if every infinite ultraproduct of members of $\mathcal{C}$ has NIP theory. As  a step in the proof of Theorem~\ref{main} we obtain the following result. Here, and throughout the paper, if $G$ is a finite group we  denote by $R(G)$ its {\em soluble radical}, that is, the unique largest soluble normal subgroup of $G$.

\begin{proposition} \label{NIPclass}
Let $\mathcal{C}$ be an NIP class of finite groups. Then there is $d=d(\mathcal{C})\in {\mathbb N}$ such that  $|G:R(G)|\leq d$ for every $G\in \mathcal{C}$.
\end{proposition}

The notion of {\em rosy} theory is a common generalisation of the notions of {\em o-minimal} theory and {\em simple} (and hence also of {\em stable}) theory. The concept was introduced in \cite{ons} and developed in \cite{adler}. We omit the definition of rosiness, but note that by \cite[Definition 0.3]{ekp}, a theory $T$ is rosy if and only if there is an independence relation $\dnf$ on real {\em and imaginary} tuples which satisfies the following natural conditions :

(i) $\dnf$ is automorphism invariant.

(ii) If $c\in \acl(aB)\setminus \acl(B)$, then $\depen{a}{B}{c}$.

(iii) If $a\dnf_B C$ and $B \cup C \subseteq D$, then there is $a' \in \tp(a/BC)$ with $a'\dnf_{B}{D}$.

(iv) There is $\lambda$ such that for any $a$, if $(B_i)_{i<\alpha}$ are sets with $B_i \subset B_j$
 whenever $i<j$ and $\depen{a}{B_i}{B_j}$ for $i<j<\alpha$, then $\alpha<\lambda$.

(v) If $B \subseteq C \subseteq D$, then $a \dnf_B D$ if and only if $a\dnf_B C$ and $a \dnf_C D$.

(vi) $C \dnf_A B$ if and only if $c \dnf_A B$ for any finite $c \subseteq C$.

(vii) $a\dnf_C b$ if and only if $b \dnf_C a$.

\noindent
A structure with an infinite descending chain of uniformly definable equivalence relations can never be rosy -- see for example the proof of Proposition 1.3 in \cite{ekp}. In particular, a field with a non-trivial definable valuation can never be rosy, and more generally a group with an infinite strictly descending chain of uniformly definable subgroups cannot be rosy. In combination
with the consequence mentioned above of the Baldwin-Saxl Theorem this yields the following, for groups. 

\begin{proposition}\cite[Corollary 1.8]{ekp}\label{centcon}
Any group definable in an NIP rosy theory has the centraliser chain condition.
\end{proposition}

By Theorem~\ref{main}, this yields immediately the following. 

\begin{corollary}
Let $G$ be a pseudofinite group with NIP rosy theory. Then $G$ has a soluble definable normal subgroup of finite index. 
\end{corollary}

We should not expect here to replace `soluble' by `nilpotent', since  examples (involving Chapuis, Simonetta, Khelif, and Zilber) are mentioned at the end of \cite{mactent} of  stable pseudofinite groups which are not nilpotent-by-finite.

Theorem~\ref{main} is proved in Section 2. In addition to Proposition~\ref{centcon}, and the classification of simple pseudofinite groups, we use the following two results. 

\begin{theorem}\cite[Wilson]{wilson2}\label{radical}
There is a formula $\psi(x)$ such that for every finite group $G$, we have
$R(G)=\{x\in G: G\models \psi(x)\}$.
\end{theorem}

\begin{theorem}\cite[Khukhro]{khukhro}\label{khuk}
There is a function $f:{\mathbb N}\to {\mathbb N}$, such that for any $d\in {\mathbb N}$,  if $G$ is
a finite soluble group with no strictly descending chain of centralisers of length $d+1$, then
$G$ has derived length at most $f(d)$.
\end{theorem}

The final section of the paper is a discussion of some possible applications of model theory to structural questions on families of finite simple groups of fixed Lie rank. There are three main sources of applications: a generalisation of the Zilber Indecomposability Theorem for groups in supersimple theories; some still-unpublished work of Ryten showing that any family of finite simple groups is an `asymptotic class', so that cardinalities of definable sets satisfy Lang-Weil-like uniformities; and information on generic types of groups in simple theories. No new results here are given. However the methods give, for example, an alternative approach to some recent advances on word maps, admittedly proving weaker results. For the Suzuki and Ree groups there is heavy dependence on a major result of Hrushovski \cite{hrushovski}.

\smallskip

{\em Acknowledgement.} We thank Sasha Borovik for drawing our attention to Theorem~\ref{khuk}.

\section{Proof of Theorem~\ref{main}.}
{\em Proof of Proposition~\ref{NIPclass}.}
Let $\mathcal{C}=\{G_i:i\in {\mathbb N}\}$ be a class of finite groups such that every non-principal ultraproduct of members of $\mathcal{C}$ has NIP theory. By Theorem~\ref{radical}, with $\psi(x)$ the formula given in that theorem, for each $i\in\omega$ we have $R(G_i)=\{x\in G_i: G_i\models \psi(x)\}$. By \L{}os's Theorem, $\psi$ defines a normal subgroup, denoted by $\psi(G)$, of any ultrapower $G$ of members of $\mathcal{C}$.

Write $\bar{G}_i:=G_i/R(G_i)$, and let $S_i:=\Soc(\bar{G}_i)$, the direct product of the minimal normal subgroups of $\bar{G}_i$. By the maximality of $R(G_i)$, each minimal normal subgroup of $\bar{G}_i$ is non-abelian and hence each $S_i$ can be written as a direct product of non-abelian simple groups. 

\smallskip

{\em Claim 1.} There is  $t\in {\mathbb N}$ such that  each  $S_i$ is a direct product of at most $t$ distinct non-abelian simple groups.

{\em Proof of Claim.} Otherwise for each $n\in {\mathbb N}$ there are infinitely many groups $G_i$ such that $S_i$ contains at least $n$ non-abelian simple
factors. If $T_1\times\ldots\times T_n$ is such a product, pick $x_j,y_j\in T_j$ with $[x_j,y_j]\neq 1$. For any $w\subset\{1,\ldots n\}$ we find $z_w$
such that $[x_j,z_w]=1$ if and only if $j\in w$ by putting $z_w=\Pi_{j\notin w}y_j$.
 It follows by \L{}os's Theorem that a non-principal ultrafilter can be chosen on ${\mathbb N}$ so that the formula $\chi(y,z)$
of form $yz\neq zy$ witnesses that $\Pi_i \bar{G}_i/\mathcal{U}$ has the independence property. Thus, as $\bar{G}_i$ is uniformly interpretable in $G_i$, the infinite group
$\Pi_{i\in {\mathbb N}} G_i/\mathcal{U}$ does not have NIP theory, a contradiction.

\smallskip

Thus, we may reduce to the case when each $S_i$ is a direct product of exactly $c$ non-abelian simple groups, namely $S_i =T_{i,1}\times\ldots \times T_{i,c}$, where each $T_{i,j}$ is non-abelian simple.

\smallskip

{\em Claim 2.} There is $e\in {\mathbb N}$  such that  any non-abelian simple subgroup of $\bar{G}_i$ has Lie rank at most $e$ (where we define the Lie rank of the alternating group ${\Alt}_n$
to be $n$, and that of the sporadic simple groups to be 1).

{\em Proof of Claim.} We argue as in the proof of Claim 1. It suffices to note that for any $n$, a sufficiently large alternating group contains a  direct product of $n$ copies of $\Alt_5$. Likewise, non-abelian classical simple groups of large rank contain many commuting copies of ${\rm PSL}_2(q)$.

\smallskip

{\em Claim 3.} Let  $\mathcal{F}$ be a family of finite simple groups of fixed Lie rank $e$. Then there is $d=d(e)\in {\mathbb N}$ such that if $K\in \mathcal{F}$ and $g,h\in K\setminus\{1\}$ then $g$ is a product of at most $d$ copies of $h$ and $h^{-1}$.

{\em Proof of Claim.} This is well-known. It follows for example from the theorem in \cite{point} that any non-principal ultraproduct of members of $\mathcal{F}$ is a group of the same Lie type over a pseudofinite field, and so is simple.

\smallskip

By Claims 2 and 3  we obtain the following: there is $b\in {\mathbb N}$  such that for 
each $i,j$ and $x_{i,j}\in T_{i,j}\setminus\{1\}$ any element of $T_{i,j}$ is a product of at most $b$ $T_{i,j}$-conjugates  of $x_{i,j}$ and $x_{i,j}^{-1}$. As $\bar{G}_i$ normalises $S_i$, it follows easily that the minimal normal subgroups of $\bar G_i$ and finally the $T_{i,j}$ themselves are uniformly definable in the $\bar{G}_i$. 
\smallskip

To complete the proof of the proposition, it suffices to show that there is $e\in {\mathbb N}$ such that $|S_i|\leq e$ for all $i$. For suppose this holds. Then $C_i:=C_{\bar{G}_i}(S_i)$ is a normal subgroup of $\bar G_i$. Since $C_i\cap S_i=1$ we have $C_i=1$. Thus, $\bar{G}_i$ embeds in $\Aut(S_i)$, so has order at most $e!$. 

So suppose for a contradiction that there is no finite upper bound on $|S_i|$.
Then by the classification of finite simple groups, there is some Lie type Chev (possibly twisted, but with the Lie rank fixed) such that the $R_{i,j}$ include arbitrarily large finite simple groups of type 
Chev.
 Relabelling if necessary, we may suppose there is a subsequence $(n_i:i\in {\mathbb N})$ of ${\mathbb N}$ 
such that  each finite simple group $R_{n_i,1}$ has Lie type Chev, and $|R_{n_i,1}|\to \infty$ as $i\to \infty$. 
We may suppose that $R_{n_i,1}$ is defined in $\bar{G}_{n_i}$ by the formula $\phi(x,\bar{a}_i)$. 

Let $\mathcal{U}$ be a non-principal ultrafilter on ${\mathbb N}$ containing $N=\{n_i:i\in {\mathbb N}\}$, and hence all cofinite subsets of $N$. Put $G=\Pi_{i\in {\mathbb N}} G_i/\mathcal{U}$, and $\bar{G}:=G/H$, where $H$ is the normal subgroup of $G$ defined by $\psi$. Then there is $\bar{a}\in\bar{G}$ such that $\phi(x,\bar{a})$ defines an infinite ultrapower of groups of type Chev,  and hence, by \cite{point}, a group of Lie type Chev over a pseudofinite field. Such a subgroup has the independence property, by the results of Ryten and  Duret  mentioned above. It follows that $\bar{G}$, and hence $G$, does not have NIP theory, a contradiction. \hfill$\Box$

\medskip

The following lemma is standard.

\begin{lemma} \label{ultra}
Let $L$ be a countable language and $M$ be a pseudofinite $L$-structure. Then there is an infinite class $\mathcal{C}$ of finite structures such that  every infinite ultrapower of members of $\mathcal{C}$ is elementarily equivalent to $M$.
\end{lemma}

\begin{proof} We may suppose that $M=\Pi_{n\in {\mathbb N}} M_i/\mathcal{U}$ where the $M_i$ are finite with $|M_i|\to\infty$ as $i\to\infty$. Let $\{\sigma_i:i\in {\mathbb N}\}$ list $\Th(M)$.
Iteratively, we find a sequence $U_0\supset U_1\supset \ldots$ of members  of $\mathcal{U}$ such
 that for each $i\in {\mathbb N}$, $U_i$ contains the $i$ smallest elements $n_{i1}<\ldots <n_{ii}$ of $U_{i-1}$, 
and such that for all $i\in {\mathbb N}$ and $j\in U_i$ with $j>n_{ii}$, $M_j\models \sigma_i$. Put $U:=\bigcap_{i\in {\mathbb N}} U_i$. Then $U$ is infinite, and by \L{}os's Theorem, $\mathcal{C}:=\{M_i:i\in U\}$ satisfies the lemma.
\end{proof}

\medskip

{\em Proof of Theorem~\ref{main}.} Let $G$ be a pseudofinite group with  NIP theory, such that every chain of centralisers has length at most $e$. Observe that there is a sentence $\tau_e$ in the language $L_g$ of groups such that for every group $H$, we have $H\models \tau_e$ if and only if every chain of centralisers in $H$ has length at most $e$.
By Lemma~\ref{ultra} there is a set $\mathcal{C}:=\{G_i:i\in {\mathbb N}\}$ and an ultrafilter $\mathcal{U}$ on ${\mathbb N}$
such that (after replacing $G$
by an elementarily equivalent group if necessary) $G=\Pi_{i\in {\mathbb N}} G_i/\mathcal{U}$, and every infinite ultraproduct of members of $\mathcal{C}$ is elementarily equivalent to $G$. It follows that $\mathcal{C}$ is an NIP class of finite groups, so by Proposition~\ref{NIPclass}  there is $d\in {\mathbb N}$ such that $|G_i:R(G_i)|\leq d$ for all $i\in {\mathbb N}$. Also
$M_i\models \tau_e$ for cofinitely many $i\in {\mathbb N}$. Hence, by Theorem~\ref{khuk}, $R(G_i)$ has derived length
at most $f(e)$ for cofinitely many $i\in {\mathbb N}$. The property that the derived length is at most $f(e)$ is first order expressible by a sentence asserting  that a certain word vanishes on a group. Thus, by \L{}os's Theorem, the normal subgroup $R(G):=\{x\in G:G\models \psi(x)\}$ is soluble of derived length at most $f(e)$, and index at most $d$ in $G$.\hfill$\Box$

\section{A pseudofinite NIP group which is not  soluble-by-finite}

We here prove the following theorem.

\begin{theorem}\label{NIPnotrosy} There is a pseudofinite group $G$ with NIP theory which is not soluble-by-finite.
\end{theorem}

If $M$ is a structure and $\phi(x_1,\ldots,x_m,y_1,\ldots,y_n)$ is a formula which 
does not have the independence property in $M$, then there is a greatest natural number $d$ such that there are distinct $\bar{a}_1,\ldots,\bar{a}_d\in M^m$ such that for each $S\subseteq \{1,\ldots d\}$ there is $\bar{b}_S\in M^n$ with, for all $i\in \{1,\ldots,d\}$,
$M\models \phi(\bar{a}_i,\bar{b}_S)\Leftrightarrow  i\in S.$ Such $d$ is called the {\em Vapnik-Cervonenkis dimension}, or {\em VC-dimension}, of the family of definable sets in the $\bar{x}$-variables determined by $\phi$ (or just of the formula $\phi$).
We note the following lemma.

\begin{lemma} \label{NIP}
Let $L,L'$ be first order languages, and let $M$ be an $L$-structure with NIP theory.
Suppose that $\{M_i:i\in I\}$ is a  set of  $L'$-structures which is uniformly interpretable in $M$
(with $I$ an interpretable set of $M$). Let $J$ be an infinite subset of $I$ and $\mathcal{V}$ a non-principal ultrafilter on $J$. Then the ultraproduct $N=\Pi_{j\in J} M_j/\mathcal{V}$ has NIP theory.
\end{lemma}

\begin{proof} It suffices to observe that the VC-dimension of any $L'$-formula $\phi(\bar{x},\bar{y})$ is uniformly bounded across the class of structures $M_i$. We leave the details as an exercise.
\end{proof}

\medskip

{\em Proof of Theorem~\ref{NIPnotrosy}.} Fix a prime $p$. It is well-known that the valued field ${\mathbb Q}_p$, and hence its valuation ring ${\mathbb Z}_p$, has NIP theory. Hence, the group $H:=\SL_2({\mathbb Z}_p)$, which is interpretable in ${\mathbb Z}_p$, also has NIP theory. Let $\mathcal{M}:=p{\mathbb Z}_p$, the maximal ideal of ${\mathbb Z}_p$. For each $k>0$ let $H_k$ be the congruence subgroup of $H$ consisting of matrices
$\begin{pmatrix}{1+a} & b \cr c & {1+d}\end{pmatrix}$
 which lie in $H$ and satisfy
 $a,b,c,d\in p^k\mathcal{M}$. Then $H_k$ is normal in $H$, and the quotient $\bar{H}_k:=H/H_k$ is finite.

Let $\mathcal{U}$ be a non-principal ultrafilter on $\omega$, and let $G$ be the ultraproduct
$\Pi \bar{H}_k/\mathcal{U}$. Then $G$ is a pseudofinite group, and is NIP by the previous lemma, since the groups 
$\bar{H}_k$ are uniformly interpretable in an NIP theory.

Note that if a group is soluble-by-finite, then so are all its subgroups and  quotients.
Therefore, in order to show that $G$ is not soluble-by-finite, we first prove the following
claim.

\smallskip

{\em Claim 1.} The group $G$ has a normal subgroup $N$ such that $G/N\cong \SL_2({\mathbb Z}_p)$.

{\em Proof of Claim.} We view the groups $\bar{H}_k$ and $G$ as structures  in the language
 $L^+:=L_g \cup \{P_i:i<\omega\}$ where the $P_i$  are unary predicates. In $\bar{H}_k$, $P_i$ is interpreted by $H_i/H_k$ for
 $i\leq k$ and by $1=H_k/H_k$ for $i>k$. Thus, the $P_i$  are interpreted by a descending chain of normal subgroups of 
$\bar{H}_k$. The
 group $G$ has by \L{}os's Theorem a corresponding strictly descending chain  $P_0^G>P_1^G>\ldots$
 consisting of normal subgroups of $G$. Put
$N:=\bigcap_{i\in\omega}P_i^G$. Compactness together with $\omega_1$-saturation of $G$ (viewed as an $L^+$-structure) yields that $G/N\cong \SL_2({\mathbb Z}_p)$.

\smallskip

To complete the proof of
Theorem~\ref{NIPnotrosy} we now note:

{\em Claim 2.} The group $\SL_2({\mathbb Z}_p)$ is not soluble-by-finite.

{\em Proof of Claim.} This must be well-known: if it were soluble-by-finite, then so
would be $\SL_2({\mathbb Z})<\SL_2({\mathbb Z}_p)$ and its quotient $PSL_2({\mathbb
Z})$, which is a free product of a cyclic group of order two and a cyclic group of order three,
and clearly not
soluble-by-finite (see \cite{Rob}, Section 6.2).
\hfill$\Box$

\begin{remark} \rm
Let $G$ be a pseudofinite NIP group which is not soluble-by-finite.  By Lemma~\ref{ultra} $G\equiv H$ for some ultraproduct $H=\Pi_{i\in {\mathbb N}} H_i/\mathcal{U}$, such that every infinite ultraproduct of the $H_i$ is elementarily equivalent to $G$. 

 The formula $\psi(x)$ defines a non-soluble normal subgroup $\psi(H)$ of finite index in $H$. By the methods of Section 2, it can be shown that   $\psi(H)$ has subgroups $N_1< N_2$ which are normal in $H$, such that $\psi(H)/N_2$ is pro-soluble (an inverse limit of soluble groups) but not soluble, and $N_1$ is the union of a chain of soluble groups but is not soluble. We have not investigated the possible structure of $N_2/N_1$. In fact, these conclusions can be shown to hold for {\em any} infinite NIP group which is a non-principal ultraproduct of distinct finite groups and is not soluble-by-finite.
\end{remark}

\section{Model theory of finite simple groups}

In this section we make some remarks about possible applications of model theory to finite group theory, via pseudofinite groups. As mentioned in the introduction, one generalisation of the notion of {\em stable} first order theory is that of {\em simple theory}. This notion was introduced by Shelah in \cite{shelah} and developed in the 1990s in \cite{kim} and \cite{kp} and further in other papers. Many ideas first appeared in \cite{ch} and in early versions of \cite{hrush2}. A convenient source, mainly used below, is \cite{wagner}. Simplicity theory is a context for an abstract theory of independence, given by `non-forking', which is less powerful than the corresponding independence theory in stability theory, but stronger than that in rosy theories. In stable theories, over a suitable base, the first order type of tuples $\bar{a}$ and $\bar{b}$, combined with the knowledge that they are independent, determines the type of $\bar{a}\bar{b}$, but this is false in general in simple theories. 

We emphasise the distinction between the group-theoretic notion of {\em simple group} and the model-theoretic notion of {\em group definable in a simple theory}. We also stress that our methods below only seem to have applications for families of finite simple groups {\em of fixed Lie rank}.

Among the simple theories are the {\em supersimple ones}, for which there are global model-theoretic notions of  {\em rank} or {\em dimension} for definable sets. We shall only deal with supersimple finite rank theories, in which all the main notions of model-theoretic rank  coincide on any definable set (though not on types). Below, we shall refer to $\SU$-rank, described later in more detail.

It can be shown  that any family of finite simple groups of fixed Lie rank is uniformly interpretable in  a family of finite fields, or (in the case of Suzuki and Ree groups) in a family of finite {\em difference fields}, that is, fields equipped with an automorphism. In fact, by  \cite[Ch. 5]{ryten}, if parameters are allowed then the groups are uniformly bi-interpretable with the (difference) fields. Thus, the groups ${\rm PSL}_3(q)$ are uniformly parameter bi-interpretable with the fields $\Ff_q$, the Ree and Suzuki groups ${}^2F_4(2^{2k+1})$ and ${}^2B_2(2^{2k+1})$ are uniformly parameter bi-interpretable with the difference fields $(\Ff_{2^{2k+1}}, x\mapsto x^{2^k})$, and the Ree groups ${}^2G_2(3^{2k+1})$ are uniformly parameter bi-interpretable with the difference fields $(\Ff_{3^{2k+1}}, x\mapsto x^{3^k})$. Now infinite ultraproducts of finite fields have supersimple SU-rank rank 1 theory -- that is, the set defined by the formula $x=x$ has $\SU$-rank 1 -- by for example \cite{cdm}. The ultraproducts of the corresponding difference fields also have supersimple $\SU$-rank 1 theory, by the results of Hrushovski \cite{hrushovski} and of Ryten (see e.g. \cite[Theorem 3.5.8]{ryten}). For the difference fields this rests on deep work from the 1990s in \cite{hrushovski}, and Hrushovski was clearly aware then of the supersimplicity of pseudofinite simple groups, and applications similar to some of those  below.

We mention three possible lines of application to finite simple groups. Some methods of this kind were used (though not for Ree and Suzuki groups), in the important paper \cite{hp1}.

 1. {\em Zilber Indecomposability.} The Irreducibility Theorem for linear algebraic groups was reworked by Zilber for groups of finite Morley rank. Other model-theoretic versions have appeared, but for us the following result of Wagner is convenient. See \cite[4.5.6]{wagner}, or, for the guise below, \cite[Remark 2.5]{er}.
 
 \begin{theorem}[Indecomposability Theorem] \label{zilberind}
 Let $G$ be a group interpretable in a supersimple finite $\SU$-rank theory, and let $\{X_i:i\in I\}$ be a collection of definable subsets of $G$. Then there exists a definable subgroup $H$ of $G$ such that:
 
 (i) $H\leq \langle X_i:i\in I\rangle$, and there are $n\in {\mathbb N}$, $\epsilon_1,\ldots,\epsilon_n\in \{-1,1\}$, and
 $i_1,\ldots,i_n \in I$, such that $H\leq X_{i_1}^{\epsilon_1}\ldots X_{i_n}^{\epsilon_n}$.
 
 (ii) $X_i/H$ is finite for each $i\in I$.
 
 If the collection of $X_i$ is setwise invariant under some group $\Sigma$ of definable automorphisms 
of $G$, then $H$ may be chosen to be $\Sigma$-invariant. 
 \end{theorem}

This has the following almost immediate application to finite simple groups. The result below can also be deduced from \cite[Theorem 1]{lns}, in combination with Theorem~\ref{asymp} below.

\begin{theorem}
Let $\C_\tau$ be a family of finite simple groups of fixed Lie type $\tau$, and let $\phi(x,y_1,\ldots,y_m)$ be a formula in the language of groups. Then there is a positive integer $d=d(\phi,\tau)$ with the following property: if  
$G\in \C_\tau$, $\bar{a}\in G^m$, and $X=\phi(G,\bar{a})$ satisfies $|X|>d$,
 then $G$ is a product of at most $d$ conjugates of the set
$X\cup X^{-1}$.
\end{theorem}

\begin{proof} Suppose that this is false, and let $\C_\tau:=\{G_i:i\in{\mathbb N}\}$. Then there is a decreasing sequence of infinite subsets 
$(I_j:j\in {\mathbb N})$ of ${\mathbb N}$ with infinite intersection $I$ such that for any 
$d\in {\mathbb N}$, and for all but finitely many $j\in I_d$, $G_j$ is not a product of at most $d$ conjugates of $X_j \cup X_j^{-1}$. Choose a non-principal ultrafilter $\mathcal{U}$ on ${\mathbb N}$ which contains the set $I$. Let $G:=\Pi_{j\in {\mathbb N}} G_i/\mathcal{U}$ and 
$X:=\Pi_{j\in {\mathbb N}} G_i/\mathcal{U}$. Then $G$ is a simple pseudofinite group so has supersimple finite $\SU$-rank theory, and $X$ is  an infinite definable subset of $G$ such that for each $d\in {\mathbb N}$, $G$ is not a product of at most $d$ conjugates of $X\cup X^{-1}$.  By Theorem~\ref{zilberind} (including the final assertion), $G$ has an infinite definable normal subgroup $H$ which is contained in a product of a bounded number of conjugates of $X\cup X^{-1}$. This is a contradiction, since by simplicity of $G$, we have $H=G$. 
\end{proof}

\medskip

Other applications of Theorem~\ref{zilberind}
were found in \cite{lmt}. In particular, it was shown in Corollary 4.11 that certain maximal subgroups (those which are not `subfield subgroups') of finite simple groups are uniformly definable in the groups, and hence, if also unbounded in order, they are `uniformly maximal' \cite[Proposition 4.2(ii)]{lmt}.

\medskip

2. {\em Asymptotic classes.} The following definition is due to Elwes \cite{e}, extending the 1-dimensional case of \cite{macs}.

\begin{definition} \label{asympclass} \rm
A class $\C$ of finite first order structures is, for some positive integer $N$, an {\em $N$-dimensional asymptotic class}, if the following holds.

(i) For
every $L$-formula $\phi(\bar{x}, \bar{y})$ where $l(\bar{x})=n$ and $l(\bar{y})=m$,
there is a finite set of pairs $D \subseteq (\{0,\ldots,Nn\}
\times {\mathbb R}^{>0}) \cup \{(0,0)\}$ and for each $(d, \mu) \in
D$ a collection $\Phi_{(d,\mu)}$ of pairs of the form $(M,
\bar{a})$ where $M \in \C$ and $\bar{a} \in M^m$, so that $\{
\Phi_{(d, \mu)} : (d, \mu) \in D \}$ is a partition of 
$\{ (M, \bar{a}): M \in \C,  \bar{a} \in M^m \}$, and
$$\big||\phi(M^n, \bar{a})| - \mu|M|^{\frac{d}{N}}\big|
=o(|M|^{\frac{d}{N}})$$ as $|M| \rightarrow \infty$ and $(M,
\bar{a}) \in \Phi_{(d,\mu)}$.

(ii) Each $\Phi_{(d, \mu)}$ is $\emptyset$-definable, that is to
say $\{ \bar{a} \in M^m : (M, \bar{a}) \in \Phi_{(d, \mu)}\}$ is
uniformly $\emptyset$-definable across $\C$.
\end{definition} 

By the main theorem of \cite{cdm}, the class of finite fields is a 1-dimensional asymptotic class, and by Theorem 3.5.8 of \cite{ryten} the classes of difference fields $(\Ff_{2^{2k+1}}, x\mapsto x^{2^k})$ and
$(\Ff_{3^{2k+1}}), x\mapsto x^{3^k})$ also form 1-dimensional asymptotic classes. The bi-interpretability results of Ryten  mentioned above now yield the following.

\begin{theorem} \cite[Ryten]{ryten}\label{asymp} If $\C$ is a family of finite simple groups of fixed Lie type, then $\C$ is an $N$-dimensional asymptotic class for some $N$.
\end{theorem}

\begin{remark}\label{imbric} \rm Let $\C=\{G_i:i\in {\mathbb N}\}$ be an asymptotic class of finite simple groups as
 above, and let $G^*:=\Pi_{i\in {\mathbb N}}G_i/\mathcal{U}$ be an infinite ultraproduct of members
 of $\C$. Let $\phi(\bar{x},\bar{y})$ be a formula with $l(\bar{x})=m$ and
 $l(\bar{y})=n$, let $\bar{a}\in (G^*)^n$ with $\bar{a}=(\bar{a}_i)/\mathcal{U}$, and suppose
 that there is $U \in \mathcal{U}$ such that
for  all $i\in U$, $\phi(G_i^{m},\bar{a}_i)$ has size
 approximately 
$\mu|G_i|^d$ (in the sense of asymptotic classes). Then it follows
 that $\SU(\phi((G^*)^m,\bar{a}))=d.\SU(G^*)$. This can be deduced from 
\cite[5.4]{e}, since $\C$ is parameter-bi-interpretable with a 1-dimensional asymptotic class (of fields or difference fields).
\end{remark}

\medskip

3. {\em Word maps.} Let $w(x_1,\ldots,x_d)$ be a non-trivial group word in $x_1,\ldots,x_d$, that is, a non-identity element of the free group $F_d$ with free basis $\{x_1,\ldots,x_d\}$. Then $w$ defines, in any group $G$, a map $w:G^d \to G$, the {\em word map} corresponding to $w$, with image denoted by $w(G)$. It is shown in \cite{larsen} that there is a function $f$ such that if $G$ is a finite simple group, $w$ is a non-trivial word,  and $\epsilon>0$, then $|w(G)|\geq |G|^{1-\epsilon}$ for sufficiently large $G$. In fact (and this could also be deduced from the last statement using Theorem~\ref{asymp}), we have: if $\C$ is a family of finite simple groups of fixed Lie type, and $w$ is a non-trivial word, then  there is $\mu>0$ such that if $G\in\C$ is sufficiently large then $|w(G)|\geq \mu|G|$.

\begin{theorem}\cite{lmt}\label{shalevtheorem}
 For any non-trivial words $w_1,w_2$ there are $N=N(w_1,w_2)$ such that if $G$ is a finite simple group with $|G|\geq N$ then $w_1(G)w_2(G)=G$. 
\end{theorem}
 
 This result is the culmination of  work in several other related papers. For example, it was shown by Shalev \cite{shalev} that if $w$ is a non-trivial word then there is $N=N(w)$ such that if $G$ is a non-abelian finite simple group with $|G|>N$ then $(w(G))^3=G$; and Theorem~\ref{shalevtheorem} was already proved for groups of fixed Lie type (other than the Ree and Suzuki groups) in \cite{ls}.

We  mention a possible alternative approach, which 
yields weaker statements than that of Theorem~\ref{shalevtheorem}, but has potential for further applications, since it depends just on the definability of $w(G)$ and its asymptotic size. For one such application, see Theorem~\ref{translatewords} below. The approach rests on the above-stated result of Larsen from \cite{larsen}, and some general model theory of groups in (super)simple theories. An advantage is  that Suzuki and Ree groups can be treated simultaneously with other families of finite simple groups with no extra work, though this rests on  the major work of Hrushovski in \cite{hrushovski}, in combination with \cite{ryten}.

First, for groups definable in simple theories there is a theory of generic types, analogous to that in stable theories,
developed by Pillay \cite{pill} and described in \cite[Sections 4.3--4.5]{wagner}. We shall consider a simple theory $T$, such that in any $M\models T$ there is an $\emptyset$-definable group $G$. Let $\dnf$ denote the relation of non-forking (i.e. independence) in simple theories: for subsets $A,B,C$ of $M$, $A\dnf_C B$ denotes that $A$ and $B$ are independent over $C$ in the sense of non-forking, that is, for any $\bar{a}$ from $A$, $\tp(\bar{a}/ B\cup C)$ does not fork over $C$. If $A$ is a set of parameters in $M\models T$, then $S_G(A)$ denotes the set of types over $A$ which contain the formula $x\in G$; that is the set of maximal consistent (with $T$) sets of formulas in the variable $x$, with parameters from $A$, which include the formula $x\in G$. Following \cite{wagner} (see Definition 4.3.2 and also Lemma 4.3.4) a type $p\in S_G(A)$ is {\em generic} if for any $b\in G$ and $a$ realising $p$ with $a\dnf_A b$, we have $ba\dnf A,b$. The group $G$ has a certain subgroup $G_A^o$ (the `connected component over $A$'), and a generic type is {\em principal} if it is realised in $G_A^o$ (where $G$ is interpreted in a sufficiently saturated model of $T$). Part (i) of the following result
 was first proved in \cite[Proposition 2.2]{psw}, and  (ii) is an immediate consequence. 

\begin{theorem}\label{3gens}
Let $T$ be a simple theory over a  countable language, $\bar{M}$ an $\omega_1$-saturated model of $T$ with a countable elementary substructure $M$, and $G$ an $\emptyset$-definable group in $\bar{M}$.
Let $p_1,p_2,p_3$ be three principal generic types of $G$ over  $M$. 

(i) There are $g_1,g_2\in \bar{M}$ such that $g_i\models p_i$ for $i=1,2$, $g_1 \dnf_M g_2$, and $g_1g_2\models p_3$.

(ii) If $r\in S_G(M)$ has realisations in $G_M^o$ then there are $a_i\in G$ with $a_i\models p_i$ (for $i=1,2,3$) such that $a_1a_2a_3\models r$. 
\end{theorem}

\begin{proof}
(i) See for example \cite[Proposition 4.5.6]{wagner}, though as phrased above one must use $\omega_1$-saturation  to find the $g_i$ in $\bar{M}$.

(ii) Choose $a_3,b\in \bar{M}$ such that $a_3\models p_3$, $b\models r$, and 
$a_3\dnf_M b$, and put $c_3:=ba_3^{-1}$. Let $p_3':=\tp(c_3/M)$. Then $p_3'$ is a generic type of $G^*$ over $M$. Indeed (repeatedly using 4.3.2 and 4.3.4 of \cite{wagner}), $\tp(a_3^{-1}/M)$ is generic, so as
$a_3^{-1}\dnf_M b$, we find $\tp(a_3^{-1}/Mb)$ is generic, so $\tp(ba_3^{-1}/Gb)$ is generic.
As $\tp(a_3^{-1}/M)$ is generic and $a_3^{-1}\dnf_M b$ we also get  $ba_3^{-1} \dnf M, b$, so $ba_3^{-1}\dnf_M b$, and this forces that $p_3'=\tp(ba_3^{-1}/M)$ is generic. Also, by the assumptions on $p_3$ and $r$, $p_3'$ has  realisations in $G^o$ so is principal.

It follows by (i) that there are $a_1,a_2\in \bar{M}$ such that $a_1\models p_1$, $a_2\models p_2$, and 
$a_1a_2=c_3$. Hence $a_1a_2a_3=b\models r$.
\end{proof}

\medskip

We also observe the following, which can be found for example in  \cite{wagner}. The $\SU$-rank on types is an ordinal-valued rank defined by transfinite induction:  for any type $p$ over $A$, $\SU(p)\geq \alpha+1$ if there is $B\supset A$ such that $p$ has a forking extension $q$ over $B$ with $\SU(q)\geq \alpha$, and for limit ordinals $\delta$, $\SU(p)\geq \delta$ if $\SU(p)\geq \beta$ for all ordinals $\beta<\delta$. If $X$ is a set defined by a formula $\phi(x,\bar{a})$ with $\bar{a}$ from $A$,  then $\SU(X)$ is the supremum (which will be  the maximum in the  finite rank theories considered here) of the $\SU(p)$ for types $p$ over $A$ containing the formula $\phi(x,\bar{a})$.

\begin{lemma}\label{rankgeneric}
Let $G$ be a group definable in a finite $\SU$-rank supersimple theory, and let $A$ be a parameter set. Then

(i) If $p\in S_G(A)$ then $p$ is generic if and only if $\SU(p)=\SU(G)$.

(ii) If $X$ is an $A$-definable subset of $G$, then $\SU(G)=\SU(X)$ if and only if some generic type $p\in S_G(A)$ contains a formula defining $X$. 
\end{lemma}

\begin{proof} (i) See \cite[p. 168]{wagner}. 

(ii) Immediate from (i) and the definition of $\SU$-rank for types and formulas.\end{proof}

\begin{theorem} \label{lst_theorem}
Let $\C_\tau$ be a family of finite simple groups of fixed Lie type $\tau$, and let 
$w_i(x_1,\ldots,x_{d_i})$ be  non-trivial words, for $i=1,2,3$.  

(i) There  is $N=N(\tau,w_1,w_2,w_3)\in {\mathbb N}$ such that if $H\in \C_\tau$ with $|H|>N$ then 
$w_1(H)w_2(H)w_3(H)=H$.

(ii) $|H\setminus w_1(H)w_2(H)|=o(|H|)$ for sufficiently large $H\in\C_\tau$.

(iii) $|w_1(H)w_2(H)|/|H|\to 1$ as $|H|\to \infty$, for $H\in \C_\tau$.
\end{theorem}

\begin{proof}
(i) Suppose that (i) is false. Then there is an infinite ultraproduct $G^*$ of members of $\C_\tau$ such that
$w_1(G^*)w_2(G^*)w_3(G^*)$ is a proper subset of $G^*$.
Also,  $G^*$ is $\omega_1$-saturated, and has a countable elementary substructure $G$. By the result of Larsen \cite{larsen} mentioned above,  there is $\mu>0$ such that if $H\in\C_\tau$ is sufficiently large then $|w_i(H)|\geq \mu|H|$ for $i=1,2,3$. 
It follows from Remark~\ref{imbric} that  $SU(w_i(G^*))=SU(G^*)$ for each $i$. Hence, by Lemma~\ref{rankgeneric}, there is for each $i=1,2,3$ a  generic type $p_i$ of $G^*$ over $G$ containing the formula
$x\in w_i(G^*)$. As  all models of $T:=\Th(G)$ are simple, a very saturated model of $T$ cannot have a proper subgroup of bounded index, so $G^*=(G^*)_M^o$ and all generic types of $G^*$ are principal.

Let $r$ be any type over $G$ realised in $G^*\setminus w_1(G^*)w_2(G^*)w_3(G^*)$. Then by Theorem~\ref{3gens}(ii), there are $a_1,a_2,a_3,b\in G^*$ such that $a_i\models p_i$ and $b\models r$ and $a_1a_2a_3=b$. In particular, $a_i\in w_i(G^*)$, so $b\in w_1(G^*)w_2(G^*)w_3(G^*)$, which is a contradiction.

(ii) Again, suppose this is false. Then by Theorem~\ref{asymp} there is $\nu>0$ and infinitely many groups $H\in \C_\tau$ such that
$|H\setminus w_1(H)w_2(H)|>\nu|H|$.
 Then, by Remark~\ref{imbric}, we may choose an infinite ultraproduct $G^*$ of members of $\C_\tau$ such that $\SU(G^*\setminus w_1(G^*)w_2(G^*))=\SU(G^*)$. Again let $G$ be a countable elementary substructure of $G^*$. By Lemma~\ref{rankgeneric} for $i=1,2,3$ there are generic types $p_i$ of $G^*$ over $G$ such that $p_1$ contains the formula $x\in w_1(G^*)$, $p_2$ contains the formula $x\in w_2(G^*)$,
and $p_3$ contains the formula $x\in G^*\setminus w_1(G^*)w_2(G^*)$. 
By $\omega_1$-saturation and Theorem~\ref{3gens}(i) there are $a_1,a_2\in G^*$ such that $a_1\models p_1$,
$a_2\models p_2$, and $a_3:=a_1a_2\models p_3$. In particular, $a_i\in w_i(G^*)$ for $i=1,2$ so 
$a_3\in w_1(G^*)w_2(G^*)$, which is a contradiction.

(iii) This is immediate from (ii).
\end{proof}

\medskip

\begin{remark} \rm
1. Part (i) above is of course just a weakening of a special case of Theorem~\ref{shalevtheorem}. Part (iii) was proved in \cite{shalev2}. We do not know whether these model-theoretic methods can yield the stronger 
assertion that if $\C$ is a family of finite simple groups of fixed Lie rank and $w_1,w_2$ are non-trivial words, then $w_1(G)w_2(G)=G$ for sufficiently large $G\in \C$. 

2. It should be  possible to strengthen the asymptotic statements in (ii), (iii), by working with tighter error terms in the definition of `asymptotic class', in the manner of \cite{cdm} rather than with the $o$-notation. More precisely, Theorem~\ref{asymp} should still hold if, 
in Definition~\ref{asympclass}, the condition
$$\big||\phi(M^n, \bar{a})| - \mu|M|^{\frac{d}{N}}\big|
=o(|M|^{\frac{d}{N}})$$
is replaced by, for some constant $c$,
$$\big||\phi(M^n, \bar{a})| - \mu|M|^{\frac{d}{N}}\big| \leq c|M|^{\frac{d}{N}-\frac{1}{2}}.$$
We have not checked this.

3. If $w(x_1,\ldots,x_d)$ is a non-trivial word, and $\C$ is a class of finite simple groups of fixed Lie type, then $w$ defines the word map $w:G^d\to G$ for $G\in \C$. Theorem~\ref{asymp} is applicable in the class $\C$ to the formula $\phi(x_1,\ldots,x_d,y)$ which says
$w(x_1,\ldots,x_d)=y$ and hence yields information on the distribution of the
solution sets, that is, on the sizes of the fibres.
\end{remark}

Finally, we stress that the proof of Theorem~\ref{lst_theorem} depends just on the fact that the sets $w_i(H)$ (for $H\in \C_t$) are uniformly definable and have cardinality
a positive proportion of $H$. This gives the possibility of further applications. For example, translates $hw(H)$ of sets $w(H)$ have the same properties. Thus, the same proof yields the following, with an analogue also of Theorem~\ref{lst_theorem}(ii), (iii). (The definition of the sets $hw(H)$ requires a parameter, but this causes no problems as, in the proof, the countable elementary submodel $G$ of $G^*$ can be assumed to include any required parameters.)

\begin{theorem} \label{translatewords}
Let $\C_\tau$ be a family of finite simple groups of fixed Lie type $\tau$, and let $w_1,w_2,w_3$ be non-trivial words. 
Then there is $N=N(w_1,w_2,w_3,\tau)$ such that if $H\in \C_\tau$ and $|H|>N$ and $h_1,h_2\in H$, then 
$$w_1(H)h_1w_2(H)h_2w_3(H)=H.$$
\end{theorem}

\bibliographystyle{amsplain}

\end{document}